\documentclass[12pt,reqno]{amsart}
\usepackage{latexsym, amsmath, amscd, amssymb, amsthm, bm, mathrsfs}
\usepackage[colorlinks=true,linkcolor=blue,urlcolor=blue,
  citecolor=blue]{hyperref}
\usepackage{amsrefs}
\usepackage[T1]{fontenc}

\usepackage{graphicx} 
\usepackage{xypic}
\usepackage{multicol}

\usepackage[centering, includeheadfoot, hmargin=1.2in, vmargin=0.8in,
  headheight=30.4pt]{geometry}
  \usepackage{xypic}

\newtheorem{lemma}{Lemma}[section]
\newtheorem{thm}[lemma]{Theorem}
\newtheorem{cor}[lemma]{Corollary}
\newtheorem{prop}[lemma]{Proposition}

\theoremstyle{definition}
\newtheorem{definition}[lemma]{Definition}
\newtheorem{remark}[lemma]{Remark}
\newtheorem{example}[lemma]{Example}

\theoremstyle{remark}

\newcommand{\mb}{\mathbf}

\newcommand{\pr}[1]{\mathbb{P}^{#1}}

\newcommand{\A}{\mathcal{A}}
\newcommand{\D}{\mathcal{D}}
\newcommand{\Z}{\mathbb{Z}}
\newcommand{\Q}{\mathbb{Q}}
\newcommand{\R}{\mathbb{R}}
\newcommand{\K}{\mathbb{K}}

\newcommand{\N}{\mathbb{N}}
\newcommand{\T}{\mathbb{T}}

\newcommand\opJ{\mathop{\rm J}\nolimits}
\newcommand\bP{{\mathbb P}}

\newcommand{\im}{\operatorname{Im}}

\newcommand{\rk}{\operatorname{rk}}

\newcommand{\Ker}{\operatorname{Ker}}

\newcommand{\into}{\hookrightarrow}

\renewcommand{\geq}{\geqslant}
\renewcommand{\leq}{\leqslant}

\newcommand{\B}{\mathcal{B}}

\newcommand{\Cayley}{\operatorname{Cayley}}

\usepackage{pgf,tikz}
\usepackage{mathrsfs}
\usetikzlibrary{arrows}

\title[Higher order selfdual toric varieties]{Higher order selfdual toric varieties}

\author[A.~Dickenstein]{Alicia Dickenstein}
\address{Alicia Dickenstein \\ Department of Mathematics, FCEN \\
Universidad de Buenos Aires \\ and IMAS - CONICET \\ Ciudad Universitaria - Pab. I
\\ C1428EGA Buenos Aires \\Argentina}
\email{\href{mailto:alidick@dm.uba.ar}{alidick@dm.uba.ar}}
\urladdr{\href{http://mate.dm.uba.ar/~alidick}{mate.dm.uba.ar/~alidick}}

\author[R.~Piene]{Ragni Piene}
\address{Ragni Piene\\Department of Mathematics\\
University of Oslo\\P.O.Box 1053 Blindern\\NO-0316 Oslo\\Norway}
\email{\href{mailto:ragnip@math.uio.no}{ragnip@math.uio.no}}
\urladdr{\href{http://www.mn.uio.no/math/english/people/aca/ragnip/index.html}
{www.mn.uio.no/math/english/people/aca/ragnip/index.html}}

\begin{document}



\begin{abstract}
The notion of higher order dual varieties of a projective variety, introduced in \cite{P83}, 
is a natural generalization of the classical notion of projective duality. In this paper we
present geometric and combinatorial characterizations of those equivariant
projective toric embeddings that satisfy higher order selfduality. 
We also give several examples and general constructions. In particular, 
we highlight the relation with Cayley-Bacharach questions and with Cayley configurations.
\end{abstract}

\maketitle

\section{Introduction} \label{sec:intro}
Let $X\subset \pr{N}$ be a projective complex algebraic 
variety of dimension $n$ over an algebraically closed field  $\K$
of characteristic zero, and fix $k \in \N$. 
 A hyperplane $H$ is said to 
be \emph{tangent to $X$ to the order $k$}  at a smooth point $x$,
when $H$ contains the $k$th osculating space to $X$ at $x$ 
(see Section~\ref{sec:higher} for a precise definition).
The {\em $k$-th dual variety} $X^{(k)}$  is the Zariski closure in the dual projective
space ${(\pr{N})}^\vee$ of all hyperplanes {tangent 
to $X$ to the order $k$} at some smooth point $x$.
In particular, the first osculating space is the tangent space 
and $X^{(1)}$ is the classical dual variety.

If the dimension $d_k$ 
of the $k$th osculating space at a general point of $X$ 
is strictly smaller than the ambient dimension $N$,  
the \emph{expected} dimension of $X^{(k)}$ equals $n +N- d_k-1$, 
which is at least $n$.
Note that the \emph{actual} dimension might be smaller than $n$.

\begin{definition}\label{def:kself}
A projective variety $X\subset \mathbb P^N$ is said to be $k$-selfdual if there exists 
a linear isomorphism $\varphi\colon \mathbb P^N \to (\mathbb P^N)^\vee$ such that 
$\varphi (X)=X^{(k)}$.
\end{definition}
In particular, if $X$ is $k$-selfdual, then $\dim X=\dim X^{(k)}=n$. 
We will be concerned with the characterization of these 
special varieties for equivariant toric embeddings. 
Toric $1$-selfdual varieties were studied in \cite{BDR11}, 
while $k$th duals  $X^{(k)}$ of projective toric 
varieties $X$ were studied in \cite{DDP11}.

An equivariantly embedded projective toric variety (not necessarily normal) is 
rationally parameterized by monomials with exponents given by
a lattice configuration:
\begin{equation*}
 \A = \{a_0, \dots, a_N\} \subset \Z^n,
\end{equation*}
besides some zero coordinates (see Proposition~1.5 in Chapter 5 of \cite{GKZ}).
We denote  by $X_\A \subset \pr{N}$ the projective toric variety rationally 
parametrized by $t \mapsto (t^{a_0}: \dots : t^{a_n})$. 
We note that  by Lemma~2.14 in \cite{BDR11}, 
$X_\A$ is non degenerate (i.e., \ not contained in any hyperplane)
 if and only if the points $a_i$ are distinct, 
and we will always assume this holds. A lattice configuration $\A$ is said to be \emph{complete}
if it consists of all the lattice points in its convex hull.  
Complete configurations $\A$ correspond to toric varieties embedded by a complete linear system.

We investigate different characterizations of higher selfdual toric varieties $X_\A$ associated to a finite lattice configuration. 
Section~\ref{sec:higher} deals with a first characterization of higher selfduality in terms of the torus action
associated to $\A$. 
\, Theorem~\ref{th:dim}, which is a generalization of Theorems~3.2 
and~3.3 in \cite{BDR11}, asserts that a non-degenerate projective toric variety $X_\A$ is $k$-selfdual if and only if
$\dim X_\A = \dim X^{(k)}_\A$ and $\A$ is $k$nap (see Definition~\ref{def:knap}). 
The main result in Section~\ref{sec:character} is Theorem~\ref{thm:eL}, 
which characterizes $k$-selfduality in combinatorial terms.  
Section~\ref{sec:surfaces}, showcases different examples in the surface case which reveal the impossibility of an exhaustive classification.
Section~\ref{sec:general} contains general constructions of higher selfdual varieties. In particular, 
we highlight the relation with Cayley-Bacharach questions~\cite{EGH} and with Cayley configurations (see Definition~\ref{def:Cayley}). 
We  show that any general configuration of $\binom{n+k}{k}+1$ points is $k$-selfdual (see Corollary~\ref{cor:general}) 
and that Veronese-Segre embeddings give smooth selfdual toric varieties in any dimension (see Theorem~\ref{th:complete}).

\medskip

\noindent{\bf Acknowledgements:}
We thank David Perkinson for sending us the reference to J. Mulliken's thesis, which was the source
of several examples. We are also grateful to Emilia Mezzetti for very useful conversations.
 A.~Dickenstein is partially supported by UBACYT
  20020100100242, CONICET PIP 112-200801-00483 and ANPCyT 2013-1110
  (Argentina).
  She also acknowledges the support of the M.~Curie
Initial Training Network ``SAGA''  that made possible her visit 
to the Centre of Mathematics for Applications (CMA) of the University of Oslo, in May 2012, 
where this work was initiated,
and to the International Center for Theoretical Physics (ICTP), where it was finished. 
R. Piene acknowledges support from the Research Council of Norway, project number 239015. 
Both authors thank the organizers of the CIMPA research school II ELGA in Cabo Frio, Brazil, 2015 
and the CMO--BIRS workshop Algebraic Geometry and Geometric Modeling in Oaxaca, 
Mexico, 2016 for the invitations and support.

\medskip

\section{Higher order duals and torus orbits}\label{sec:higher}

The main result of this section is Theorem~\ref{th:dim}, where we give a first characterization of
higher order selfdual toric varieties $X_\A$ deduced from 
the description of $X_\A^{(k)}$ in terms of  the torus action 
defined by $\A$ (see~\eqref{eq:action},~\eqref{eq:actiondual}).

\subsection{Higher duals of projective varieties}
We recall here some basic facts and we refer the reader 
to \cite{BDR11}, \cite{DDP11} for more details.

Let $\iota\colon X\into \pr{N}$ be an 
embedding of a complex non degenerate irreducible algebraic
variety of dimension $n$. 
Let  $(x_1,\dots ,x_n)$ be a local system of coordinates around  a smooth point $x \in X$, 
with maximal ideal ${\mathfrak m}_x=(x_1,\dots  ,x_n)$ 
in the local ring ${\mathcal O}_{X,x}$. Let $\mathcal
L:=\iota^*({\mathcal O}_{\pr{N}}(1)).$ 
 The quotient vector space $\mathcal L/{\mathfrak m}_{x}^{k+1}\mathcal L$
is the fibre at $x$ of the $k$-th principal parts (or jet)
sheaf ${\mathcal P}_{X}^k(\mathcal L)$. 
The  $k$-th jet map (of coherent sheaves) 
$j_k: {\mathcal O}_X^{N+1}\to  {\mathcal P}_{X}^k(\mathcal L)$
is given fiberwise by the linear map  $
j_{k,x}$ induced by the map of ${\mathcal O}_{X}$-modules
$\mathcal L\to \mathcal L/{\mathfrak m}_x^{k+1}\mathcal L$, which sends
 $s \in {\mathcal L}_x$ to its Taylor series expansion 
up to order $k$  with respect to the local coordinates $x_1,\ldots,x_n$.
Thus, $\pr{}(\im(j_{1,x}))=\pr{}({\mathcal P}_{X}^1(\mathcal L)_x)={\T}_{X,x}$ 
is the embedded tangent space at the point $x$. 
More generally,
the linear space $\pr{}(\im(j_{k,x}))$ is called the
\emph{$k$-th osculating space} of the embedded variety $X$  at $x$ and it is denoted by ${\T}^{k}_{X,x}$. 
The  dimension of the $k$th osculating spaces is at most $\binom{n+k}k-1$.

The variety $X$ is called \emph{generically  $k$-jet spanned}
if equality holds for almost all smooth points $x\in X$. Let $X_{k{\rm -cst}}$ 
denote the open dense subset of $X$ where 
the rank of $j_k$ is constant; denote this rank by $d_k+1$.

\begin{definition} \label{def:Xk}
A hyperplane $H$  is tangent to $X$ to order $k$  at a point
 $x\in X_{k{\rm -cst}}$ if $\T_{X,x}^{k}\subseteq H$.
The $k$-th dual variety $X^{(k)}$ is 
\begin{equation}\label{eq:Xk}
X^{(k)}:={\rm closure}\{H\in ({\pr{N}})^\vee\; |\, H\supseteq \T^k_{X,x}
\text{ for some } x\in X_{k{\rm -cst}}\}.
\end{equation}
\end{definition}

\noindent It follows that $X^{(k)}$ 
is the closure of the image of the map
\begin{equation}\label{eq:mapkdual}
 \mathbb P((\Ker  j_k)^\vee|_{X_{k{\rm -cst}}})
\to ({\mathbb P^N})^\vee,
\end{equation}
and hence is irreducible.
The higher order dual varieties $X^{(k)}$ 
for $k \ge 2$ are contained in the singular locus of the dual variety $X^\vee=X^{(1)}$. 
\medskip

\subsection{Higher duals of toric varieties and $k$nap configurations}
Consider a lattice configuration
$ \A = \{a_0, \dots, a_N\} \subset \Z^n$
 and let $X_\A\subset \mathbb P^N$ be the corresponding toric variety.
The variety $X_\A$ is an affine invariant of the 
configuration $\A$ by Proposition II.5.1.2 in \cite{GKZ}  (see also Section 2 in \cite{BDR11}) 
and the dimension of $X_\A$ equals the dimension of the affine span of $\A$.
Thus we can replace, without 
loss of generality, our configuration by the affinely isomorphic lattice configuration 
\begin{equation}\label{eq:A}
\{(1,a_0), \dots, (1,a_N)\}
\subset \Z^{n+1}.
\end{equation}
We will assume that the lattice configurations $\A$ we consider are of the form~\eqref{eq:A}.
We will denote by $A$ the  integer matrix of size $(n+1)\times (N+1)$ with columns
$\A=\{(1,a_0), \dots, (1,a_N)\}$, which has rank $n+1$. 
Up to replacing $\Z^{n+1}$ by the lattice spanned by
the points in $\A$, we can assume that $\Z \A = \Z^{n+1}$, or equivalently, that the greatest
common divisor of the maximal minors of $A$ equals $1$. Thus, we will assume
that $\dim(X_\A)=n$.

\begin{definition}\label{def:Ak}
Given any matrix $A$ as above, denote by $\mb v_0 =(1, \ldots, 1), \mb v_1,
 \dots, \mb v_n \in \Z^{N+1}$ the row vectors of $A$.
For any $\alpha \in \N^{n+1}$, 
denote by $\mb v_\alpha \in \Z^{N+1}$ the vector obtained as the evaluation
of the monomial $x^\alpha$ in the points of $\A$, that is, the \emph{coordinatewise product} of 
($\alpha_0$ times the row vector $\mb v_0$)  times ($\alpha_1$ times the row vector $\mb v_1$)
 times  $\ldots$ times ($\alpha_n$ times the row 
vector $\mb v_n$).  For any $k$,  we define the associated matrix $A^{(k)}$ as follows. 
Order the vectors $\{\mb v_\alpha : |\alpha| \le k\}$ by degree and then by lexicographic
order with $0 > 1 > \dots > n$, and let  $A^{(k)}$ be the $\binom{n+k}k \times (N+1)$ 
integer matrix with these  rows.
As $\mb v_0 =(1, \ldots, 1)$,  the first $n+1$ row vectors of $A^{(k)}$ 
are just the row vectors $\mb v_0, \dots, \mb v_n$ of $A$.
\end{definition}

\begin{example}
For $k=2$, we get

{\tiny
\begin{equation} \label{eq:matrixA2}
A^{(2)} \, = \, \left(
\begin{array}{ccc} 
{} &\mb v_{(2,0,\ldots,0)} &{}\\
{} &\mb v_{(1,1,0,\ldots,0)} &{}\\
{}& \vdots& {}\\
{} & \mb v_{(1,0,\ldots,0,1)} & {}\\
{} & \mb v_{(0,2,0,\ldots,0)} & {} \\
{} & \mb v_{(0,1,1,0,\ldots, 0)} & {} \\
{} & \vdots & {} \\
{} & \mb v_{(0,\ldots, 0,1,1)} \\
{} & \mb v_{(0,\ldots,0,2)} & {}
\end{array}\right).
\end{equation}
}

For example, if $n=1$ and $\mb v_1=(0,1,\ldots,N)$, then

\[{ A^{(2)}}=\left(
\begin{array}{cccccc} 1&1&1&1& \cdots &1\\
0&1&2&3& \cdots & N\\
0& 1& 4&9& \cdots & N^2
\end{array}\right).
\] 

If instead, $n=2$ and 
\[ A = \left(
\begin{array}{cccc}
1&1&1&1 \\
0&1&0&1 \\
0&0&1&1
\end{array} \right)
\]
defines the Segre embedding of $\pr{1} \times \pr{1}$ in $\pr{3}$,  the matrix
$A^{(2)}$ has $6$ rows, and (avoiding repeated rows) 
it is affinely equivalent to the lattice configuration read in the columns of the following matrix:
\[A'^{(2)} =
\left(
\begin{array}{cccc}
1&1&1&1 \\
0&1&0&1 \\
0&0&1&1 \\
0&0&0&1
\end{array} \right).
\]
\end{example}
\medskip

A lattice configuration  $\A$ with cardinality $N+1$ defines a torus action of the $n$-torus
on $\pr{N}$ as follows:
\begin{equation}\label{eq:action}
{\mb t} *_\A {\mb x} = ({\mb t}^{{a}_0}\,x_0:\cdots:{\mb t}^{{a}_N}\, x_N).
\end{equation}
Then, $X_\A={\rm closure}({{\rm Orb_{*_\A}}({\mb 1})})$ is the closure of the orbit of the
point ${\mb 1}= (1, \dots,1)$. 
By considering a (vector space) basis in $\K^{N+1}$ and its dual basis in
$(\K^{N+1})^\vee$, we will identify $\mathbb P^N = {\mathbb P}(\K^{N+1})$ 
with ${\mathbb P}((\K^{N+1})^\vee)= (\mathbb P^N)^\vee$. 
The torus action~\eqref{eq:action} defines an action in the dual space, which is 
given by
\begin{equation}\label{eq:actiondual}
{\mb t}*^\vee_\A{\mb y}= ({\mb t}^{-{a}_0}\, y_0:\cdots:{\mb t}^{-{a}_N}\, y_N).
\end{equation}
If  for any $\mb t$ in the $n$-torus, we denote by $\mb{\frac 1 t}$ its coordinatewise inversion, 
 for any $i=0, \dots, N$,
we have that   ${\mb  t}^{-{a}_i} = ({\mb {\frac 1 t}})^{{a}_i}$.

Given a lattice configuration $\A$ as in \eqref{eq:A} and $k \in \N$, 
the projectivization of the rowspan of $A^{(k)}$  depends on the toric variety $X_\A$ and not on
the choice of the matrix $A$ (and associated matrix $A^{(k)}$). In fact, the 
projectivization  of the rowspan of $A^{(k)}$ equals the $k$-th osculating space $\T^k_{X_\A,{\mathbf 1}}$.
Thus, the rank of $A^{(k)}$ equals the generic rank $d_k+1$ of the $k$-th jet map.
We will denote the dimension of
$\Ker A^{(k)}$ by $c_k$:
\begin{equation}\label{eq:ck}
c_k \, := \, \dim \Ker  A^{(k)} =(N+1)-(d_k+1)= N - d_k.
\end{equation}

\begin{definition}\label{def:knap}
Given a lattice configuration $\A$ as in \eqref{eq:A} and $k \in \N$, 
we say that $\A$ (or the matrix $A$) is $k$nap if
no vector $e_i$, $ i =0,\dots, N$,  in the canonical basis of $\R^{N+1}$ lies in the rowspan of the
matrix $A^{(k)}$, i.e., if the configuration of columns of $A^{(k)}$ is  \emph{not a pyramid} over
one of its points. Clearly, if $\A$ is $k$nap, then it is $k'$nap for all $k'\le k$.
\end{definition}

We will denote by $T_N^\vee$ the torus of $(\mathbb P^N)^\vee$, that is, the open set
formed by the points with all nonzero coordinates.

\begin{lemma}\label{lem:knap}
Let $\A$ be a lattice configuration as in \eqref{eq:A} and $k \in \N$.
Then, the following statements are equivalent:
\begin{enumerate}
 \item  \label{it1} $\A$ is $k$nap.
 \item \label{it2} There exists a point in $\Ker A^{(k)}$ 
 with all nonzero coordinates. That is, the projective linear space
${\mathbb P}({\Ker A^{(k)}})$ meets the torus $T_N^\vee$.
  \item \label{it3} For any $i \in 0, \dots, N$, it is not possible to find a polynomial 
  $Q \in \Q[x_1, \dots, x_n]$
of degree at most $k$ such that $Q(a_j) = 0$ for all $j \not= i$ and $Q(a_i) \neq 0$.
\end{enumerate}
\end{lemma}

\begin{proof}
The variety $\mathbb P({\Ker A^{(k)}})$  meets the torus if and only if it does not 
lie in the union of the coordinate hyperplanes. But as it is irreducible, 
this happens if and only if it does not lie in a single
coordinate hyperplane $\{ y_i=0\}$, for  some $i=0, \dots, N$. Thus condition (b) is clearly equivalent to
condition (a). To prove the equivalence with condition (c), it is enough to observe that
any vector in the $\mathbb Q$-rowspan of $A^{(k)}$ is of the form 
$(Q(a_0), \dots, Q(a_m))$ for a polynomial $Q$ of degree at most $k$.
\end{proof}

In general, let $J$ be the \emph{maximal} set of indices such that
${\mathbb P}({\Ker A^{(k)}})$ is contained in the coordinate space
$P_{N,J}:=\{x \in (\pr{N})^\vee \; |\, x_j=0 \text{ for all } j \in J\}$. Denote by
$T_{N,J}^\vee: =\{x \in P_{N,J} \; |\,, x_i \not= 0 \text{ for all } i \notin J \}$ the torus of $P_{N,J}$.
In particular, when $\A$ is $k$nap, $J=\emptyset$ and $T_{N,\emptyset}^\vee=T_N^\vee$.

\begin{prop}\label{prop:action}
Let $k \in \N$ and $\A$ a lattice configuration as in~\eqref{eq:A}.  
Then the $k$th order dual variety can be written as
\begin{equation}\label{eq:orb}
X_\A^{(k)}={\rm closure}\left(
{\bigcup {\rm Orb}_{*^\vee_\A}({\bf y})}\right),
\end{equation}
where the union is taken over ${{\bf y}\in {\mathbb P}({\Ker A^{(k)}}) \cap T_{N,J}^\vee}$.
In particular, $X_\A^{(k)}$ can be rationally parameterized, and it is 
nondegenerate if and only if $\A$ is $k$nap.
\end{prop}

\begin{proof}
As we pointed out, the projectivization  of the rowspan of $A^{(k)}$ equals the $k$-th osculating space $
\T^k_{X_\A,{\mathbf 1}}$.
The osculating spaces at the points in 
the torus of $X_\A$ are translated
by the action~\eqref{eq:action}. 
We deduce from~\eqref{eq:mapkdual} (cf. also Remark 2.20 in \cite{DFS07})  
that the $k$-th dual variety in $(\mathbb P^N)^\vee$ 
can be described as in~\eqref{eq:orb}. 

Since $\Ker A^{(k)}$ is a linear subspace, it can be (linearly) parameterized,
and hence the union of orbits in~\eqref{eq:orb} admits the following rational parameterization.
Let $\nu^{1}, \dots, \nu^{c_k}$ be a $\Z$-basis of the kernel of $A^{(k)}$  and consider
the $c_k$-dimensional ``row vectors'' $b_i:=(\nu^{1}_i, \dots, \nu^{c_k}_i), i=0, \dots, N$. 
Then, the map $\varphi: (\pr{c_k-1})^\vee \to {\mathbb P}({\Ker A^{(k)}})$ defined by
\begin{equation}\label{eq:varphi}
\varphi (\lambda) =(\langle b_0, \lambda \rangle, \dots, \langle b_N, \lambda \rangle) 
\end{equation}
is a parameterization
of ${\mathbb P}({\Ker A^{(k)}})$ (here  $\langle b_i, \lambda \rangle$ 
denotes the sum $\sum_{j=1}^{c_k} b_{ij} \lambda_j$).
It follows that  $X_\A^{(k)}$ can be rationally parame\-teri\-zed by 
sending $(\lambda,{\bf t})$, with ${\bf t}$ in the $n$-torus $(\K^*)^n$ 
of points in $\K^n$ with all nonzero coordinates, to:
\begin{equation}\label{eq:par}
(\langle b_0, \lambda \rangle \, {\bf t}^{-a_0}, \dots, \langle b_N, \lambda \rangle \,  {\bf t}^{-a_N}).
\end{equation}

By equality~(\ref{eq:orb}), $X_\A^{(k)} \subset P_{N,J}^\vee$. 
Moreover, by Lemma~\ref{lem:knap},  $J$ is empty if and only if
$\A$ is $k$nap. So, when $\A$ is not $k$nap it is clear that $X_\A^{(k)}$ 
is degenerate, and when $\A$ is $k$nap
we have 
that  $X_\A^{(k)}$ is nondegenerate because we are assuming 
that all weights $a_i$ are different and so no orbit can lie in
a linear space by by Lemma~2.14 in \cite{BDR11}.
\end{proof}

\subsection{First characterization of higher selfduality for toric embeddings}

The following result is a generalization of Theorems~3.2 and~3.3 in \cite{BDR11}.  
We will avoid the subscripts to indicate the torus action, when this is clear from the context.

\begin{thm}\label{th:dim}
 Let $\A$ be a lattice configuration as in \eqref{eq:A}.
Then, the following statements are equivalent:
\begin{enumerate}
 \item  \label{it4} $X_\A$ is $k$-selfdual.
 \item \label{it5} $\dim X_\A= \dim X_\A^{(k)}$ and $\A$ is $k$nap.
\item \label{it6} There exists a point $\bf p$  in the torus of $(\pr{N})^\vee$ such that
$X_\A^{(k)} = {\rm closure}({\rm Orb}_{*^\vee_\A}({\bf p}))$.
\end{enumerate}
\end{thm}

\begin{proof}
Assume  (a) holds. Then the equality of the dimensions in  (b) is evident.
In case $\A$ is not $k$nap, $X_\A^{(k)}$ is degenerate by Proposition~\ref{prop:action}. 
As $X_\A$ is nondegenerate, there is no linear isomorphism $\varphi$ such that $\varphi(X_\A)=X_\A^{(k)}$.

Assume now that  (b) holds. By Lemma~\ref{lem:knap} there exists a point 
${\bf p} \in {\mathbb P}({\Ker A^{(k)}}) \cap T_N^\vee$.
Then, $X_\A = {\rm closure} ({\rm Orb_{*_\A}}({\bf 1}))$ is isomorphic to 
${\rm closure}({\rm Orb_{*\vee_\A}}({\bf p}))$
 by the diagonal linear isomorphism given by coordinatewise
product by ${\bf p}$. Moreover, we deduce from \eqref{eq:orb} that
$${\rm Orb_{*\vee_\A}}({\bf p}) \subset X_\A^{(k)}=
{\rm closure}\left({\bigcup{\rm Orb_{*\vee_\A}}({\bf y})}\right),$$
where ${{\bf y}\in
 {\mathbb P}({\Ker A^{(k)}})\cap T_N^\vee}$.
As the dimensions agree and both varieties are irreducible, it follows that  (c) holds  

Finally, if (c) holds, then
the closure of the orbit of ${\bf p}$ under $*^\vee_\A$ has the same dimension as
$X_\A = {\rm Orb}_{*_\A}({\bf 1})$ and they are  isomorphic, which implies that $X_\A$ is $k$-selfdual.
\end{proof}

\begin{example}
To illustrate the fact that $k$nap-ness is needed in (b), here is an example with 
$\dim X_\A=\dim X_\A^{(k)}=2$, $\A$ not $k$nap, and $X_\A$ is not $k$-selfdual, for $k=1$. 
Consider the lattice point configuration $\A$ giving the following matrix:

\[A=\left(\begin{array}{ccccc}
1&1&1&1&1\\
1&0&0&0&0\\
0&0&1&2&3
\end{array}\right)
\]
This matrix is obviously a pyramid (so it is not $1$nap). In fact, $X_\A\subset  \pr{4}$ 
is the cone over a twisted cubic curve in a $ \pr{3}\subset  \pr{4}$ 
with vertex a point outside the $ \pr{3}$. The hyperplanes tangent to $X_\A$ 
are the hyperplanes containing this vertex and a tangent line to the twisted cubic. 
Thus the dual variety $X_\A^\vee$ is a surface contained in a hyperplane, 
hence degenerate, and therefore cannot be linearly equivalent to $X_\A$. 
Observe that $X_\A^\vee$ is not a toric variety. In fact, if we identify the hyperplane
$P_{0}$ with $(\pr{3})^\vee$,  $X_\A^\vee$ gets identified with the dual of the twisted cubic. 
We generalize this observation in Lemma~\ref{lem:DJ}.
\end{example}

\section{Combinatorial characterization of {$k$}-selfdual toric varieties}\label{sec:character}

In this section we characterize $k$-selfdual toric varieties $X_\A$ in combinatorial terms. 
Our main result is Theorem~\ref{thm:eL}.
In Proposition~\ref{prop:rCayley} we show  that when the kernel of
the associated matrix $A^{(k)}$ has dimension $c_k=1$, $k$-selfduality 
is automatic provided $\A$ is $k$nap, and we give some consequences of $k$-selfduality.
We start by recalling the definition of Cayley configurations.

\begin{definition}\label{def:Cayley}
A configuration $\A \subset \Z^{r+d}$ is said to be $r$-Cayley if there exist $r$ lattice
configurations $\A_1, \dots, \A_r \subset \Z^d$ such that $\A$ is affinely isomorphic to
$$\Cayley (\A_1,\dots,\A_r) = e_1 \times \A_1 \cup \dots \cup e_r \times \A_r,$$ 
where $\{e_1, \dots, e_r\}$ denotes the
canonical basis in $\Z^r$. 
\end{definition}

Note that we do not require the $\A_i$ to be non degenerate, i.e., 
possibly $\A_i\subset \mathbb Z^e\subset \mathbb Z^d$, with $e< d$. 
When all configurations $\A_i$ equal a given lattice configuration $\B$,
 then $X_{\Cayley (\A_1,\dots,\A_r)}$ is isomorphic to
the product $\pr{r-1} \times X_\B$.

\begin{remark}\label{rem:1Cayley}
Any $r$-Cayley configuration $\A$  lies in an affine hyperplane off the origin since
$e_1 \times \A_1 \cup \dots \cup e_r \times \A_r$ lies in the hyperplane defined by the sum of the first
$r$ coordinates equal to $1$. Modulo an affine isomorphism, we can assume that $\A$ lies
in the hyperplane defined by $x_1=1$. As we remarked before, 
without any loss of generality, any lattice configuration  is
of the form~\eqref{eq:A}, that is,  any $\A$ is  $1$-Cayley. 
Starting with $r=2$, the condition of being $r$-Cayley imposes serious combinatorial constraints;
in particular, all points in $\A$ need to lie in two parallel hyperplanes. Note also that any 
$r$-Cayley configuration is also an $r'$-Cayley configuration for any $r'\le r$. 
\end{remark}

Recall that given a lattice configuration $\A$ as in \eqref{eq:A} and $k \in \N$, we denote the dimension of
$\Ker A^{(k)}$ by $c_k$~\eqref{eq:ck}.
Given a $\Z$-basis of this kernel $\nu^{1}, \dots, \nu^{c_k}$, we denote by 
 $b_i:=(\nu^{1}_i, \dots, \nu^{c_k}_i), i=0, \dots, N$ the $c_k$-dimensional ``row vectors''.

\begin{definition}\label{def:eL}
Given a line $L$ through the origin in $\R^{c_k}$, we define the 
$(0,1)$-vector $e_L \in \R^{N+1}$ by
the condition that $(e_L)_i = 1$ if and only if $b_i \in L$.
\end{definition}

The vectors $e_L$ are independent of the choice of basis of $\Ker A^{(k)}$.

\begin{thm}\label{thm:eL}
The projective toric variety $X_\A$ is $k$-selfdual if and only if 
$\A$ is $k$nap and the vectors $e_L$ lie in the rowspan of $A$ for
each line $L$ through the origin in $\R^{c_k}$.

\smallskip
More explicitly, let $L_1, \dots, L_r$  be the lines containing some of the vectors
$b_i$, and for any $j=1, \dots,r$, set $\Gamma_j :=\{ i : b_i \in L_j \}\subseteq \{0, \dots, N\}$. 
Then, $r \ge c_k$ and  $X_\A$ is $k$-selfdual if and only if 
$\A$ is $k$nap and $\A$ is $r$-Cayley with respect to this partition of $\{0, \dots, N\}$.

\smallskip
In particular, 
\begin{equation}\label{eq:bi}
\sum_{\ell \in \Gamma_j} b_\ell =0 \quad \text{ for any } \, j=1, \dots, r.
 \end{equation}

\end{thm}

\begin{proof}
Note that the condition that the vector $e_L$ lies in the rowspan of $A$ is equivalent to
the fact that $\sum_{b_i \in L} v_i=0$ for any vector $v$ in $\Ker A$.

We recall a basic result from the theory of toric ideals \cite{BSbook}.
Let $(y_0 : \dots : y_N)$ be homogeneous coordinates in ${\pr{N}}$ 
and write any vector $v\in \Z^{N+1}$ as the difference
of two non negative integer vectors with disjoint support $v = v^+ - v^-$. 
It is well known  that a projective variety is of the form 
${\rm closure}({\rm Orb}_{*^\vee_\A}({\bf p}))$ if and only
 if it is cut out by the following binomial equations:
\begin{equation}\label{eq:toric}
{\bf p}^{v^-}y^{v^+} - {\bf p}^{v^+}y^{v^-}=0, \text{ for all }  v \in \Ker A.
\end{equation}

Assume $X_\A$ is $k$-selfdual. Then, it follows from
Theorem \ref{th:dim} that $\A$ is $k$nap and 
there exists ${\bf p} =(p_0 : \dots:p_N)\in {T}_{N}^\vee$ such that $X_\A^{(k)} = 
{\rm closure}({\rm Orb}_{*^\vee_\A}({\bf p}))$.
We substitute the rational parametrization $y_i = 
 \langle b_i, \lambda \rangle \, {\bf t}^{-a_i}, i =0, \dots, N$, 
from~\eqref{eq:par} into equations~\eqref{eq:toric}.
Then, the ${\bf t}$ variables get cancelled and for any $v \in \Ker A$ we have the
following polynomial identity in the variables $\lambda$:
\begin{equation}\label{eq:pari}
{\bf p}^{v^-} \prod_{v_i > 0}  \langle b_i, \lambda \rangle^{v_i} = 
{\bf p}^{v^+} \prod_{v_i < 0}  \langle b_i, \lambda \rangle^{-v_i}.
\end{equation}
The polynomials on both sides of~\eqref{eq:pari}  
must have the same irreducible factors to the same powers. 
Clearly, $\langle b_i, \lambda \rangle$ and $\langle b_k,\lambda \rangle$ are associated
irreducible factors if and only if $b_i$ and $b_k$ are collinear
vectors. For any line $L_j$, let $\beta_j$ be one of the two integer generators of $L_j$ and for any
$i \in \Gamma_j$ write $b_i = \mu_{ij} \beta_j, \mu_{ij} \in \Z \setminus \{0\}$.
Then, for any $v \in \Ker A$, the rational function
\[\prod_{i\in \Gamma_j}  \langle \beta_j, \lambda \rangle^{v_i} = 
\langle \beta_j, \lambda \rangle^{\sum_{i\in \Gamma_j}  v_i} \]
must be constant, which implies that 
\begin{equation}\label{eq:vi}
\sum \limits_{i \in \Gamma_j} v_{i}=0, \, \text{ for any } j=1, \dots,r.
\end{equation} 
As we remarked, this is equivalent to the fact that the
vectors $e_{L_j}$ lie in the rowspan of $A$ for any $j=1, \dots,r$. 
This means that the partition of $\A$ given 
by the subsets $\Gamma_j, j=1,\dots, r,$ gives $\A$ an $r$-Cayley structure.
As $\Ker A^{(k)} \subseteq \Ker A$, 
we get that the vectors $v= \nu^{1}, \dots, \nu^{c_k}$ also satisfy~\eqref{eq:vi},
and we deduce that the sum of the row vectors $b_i, i\in \Gamma_j$, 
satisfies~\eqref{eq:bi}, as all its coordinates are equal to zero.
Note also that since the vectors $b_i$ span $\Z^{c_k}$, 
they must lie in at least $c_k$ different lines, i.e., $r \ge c_k$.

Assume now that the vectors $e_{L_j}$ lie in the rowspan of $A$ for
each line $L_j$ through the origin in $\R^{c_k}$, or equivalently, that~\eqref{eq:vi} holds, for any
$j=1, \dots, r$ and any $v \in \Ker A$.  
As $\A$ is $k$nap, there exists a vector
${\bf p} \in {\mathbb P}({\Ker A^{(k)}}) \cap T_N^\vee$. We write it as
 ${\bf p} = \varphi(\lambda_0)$ with $\lambda_0 \in(\pr{c_k-1})^\vee$ 
and $\varphi$ the linear map defined in~\eqref{eq:varphi}.
Then, it is straightforward to verify that for any $\lambda$ with 
$\varphi(\lambda) \in {\mathbb P}({\Ker A^{(k)}}) \cap T_N^\vee$, 
\begin{equation}\label{eq:mu}
\prod_{i}  \langle b_i, \lambda \rangle^{v_i} = \mu^* 
 \left(\prod_{j=1}^r \langle \beta_j, \lambda \rangle^{\sum_{i \in \Gamma_j} v_i}\right)
= \mu^*,
\end{equation}
where the nonzero constant $\mu^*$ is equal to the product 
$\left(\prod_{j=1}^r \prod_{i \in \Gamma_j} \mu_{ij}^{v_i}\right)$.
But in particular,~\eqref{eq:mu} holds for $\lambda_0$, that is, $\mu^*={\bf p}^v$. 
Therefore, the binomial equations in ~\eqref{eq:pari} hold, as wanted.
\end{proof}

Given a lattice configuration $\A$,  we have seen in Section~\ref{sec:higher} that
the projectivization  of the rowspan of $A^{(k)}$ equals the $k$-th osculating space $
\T^k_{X_\A,{\mathbf 1}}$. Recall that the embedded toric variety
$X_\A$ is generically $k$-jet spanned  when the rank $d_k+1$ 
of $A^{(k)}$ equals $\binom{n+k}k$, i.e., if $d_k=\binom{n+k}k-1$.
Also,  $\A$ is $k$nap if and only if there is an element of the kernel of $A^{(k)}$ which lies in the torus.
We easily deduce the following restrictions.

\begin{prop} \label{prop:rCayley}
Let $\A$ be a lattice configuration as in~\eqref{eq:A}  and $k \in \N$.
Then:
 \begin{itemize}
\item[(i)] If $\A$ is $k$nap and $c_k=1$, then $X_\A$ is $k$-selfdual and $|\A|\leq \binom{n+k}k +1$.
\item[(ii)]  If  $X_\A$ is $k$-selfdual, then $\A$ is $c_k$-Cayley. 
\item[(iii)]  If  $X_\A$ is $k$-selfdual for $k\geq 2$ and $c_k\ge 2$, 
then $X_\A$ is  not generically $k$-jet spanned.
\end{itemize}
\end{prop}

\begin{proof}
To prove (i), assume $\A$ is $k$nap and
$c_k=1$. Then it follows from Theorem~\ref{thm:eL} that $X_\A$ is $k$-selfdual. 
In fact, since $\mathbb R^{c_k}=\mathbb R$, 
there is only one line $L=\mathbb R$, and $e_L=(1,\dots ,1)$ is in the rowspan of $A$.  
The inequality follows from the fact 
that $A^{(k)}$ has $|\A|$ columns and $\binom{n+k}k$ rows.

Item (ii) is proved in Theorem~\ref{thm:eL} .

In case $c_k \ge 2$,  there exist at least two (nonzero) $(0,1)$-vectors
with \emph{disjoint} support lying in the row span of $\A$, 
for instance the vectors $e_{L_1}$, $e_{L_2}$. As their
coordinatewise product is the zero vector, we see that $\rk A^{(k)}$ 
cannot be maximal  for any $k \ge 2$. This proves (iii).
\end{proof}

Assume $X_\A\subset \mathbb P^N$ is $k$-selfdual and of dimension $n$, 
then by Proposition \ref{prop:rCayley}, $X_\A$ is $c_k$-Cayley. 
Clearly $c_k\le n$.  Since $c_k=N-d_k$, it follows that $d_k\ge N-n$. 
This gives a lower bound for the dimension $d_k$ of the $k$th osculating space at a general point of $X_\A$.


\section{Higher selfdual surfaces}\label{sec:surfaces}

In the case of surfaces, Corollary~\ref{prop:rCayley} implies that for any
$k\ge 2$,  if the surface $X_\A$ is $k$-selfdual, either $c_k = 1$ 
or $X_\A$ is a scroll (all points lie on two parallel lines,
which are at lattice distance $1$ if $\A$ spans $\Z^2$).  

In this section we give several examples which are not scrolls (so necessarily $c_k=1$), 
which are partly inspired by the BA Thesis of Mulliken~\cite{M95}.
Examples of $k$-selfdual toric varieties can be extracted 
from studies of failure of the condition of $k$-jet spannedness as in~\cite{Pe00}.

It is clear by the examples we present that there is no hope in classifying 
smooth $k$-selfdual toric varieties $X_\A$, even with $c_k=1$. 
Instead,  it could be feasible to characterize $k$-selfdual toric varieties 
when  $X_\A$ is smooth and $\A$ is complete (i.e., it consists 
of all lattice points in its convex hull). 
This is related to the classification of minimal smooth Togliatti systems in~\cite{MM16},
restricted to the complete case.

To understand the difficulty in finding a complete classification,  
we refer also to~\cite{Betal15} and~\cite{L10}. In~\cite{Betal15}, the authors 
prove that once the dimension  $n$ of the variety and $N$ of the ambient space are fixed, 
there are only finitely many smooth, toric varieties corresponding to complete configurations. A
 full list of all possible configurations $\A$ is given in case of surfaces and threefolds, with $N$ at most eleven.

\subsection{Togliatti surface and generalizations}\label{ssec:Togliatti}

A classical example, going back to Togliatti \cite{Tog29}, of a smooth surface such that almost all second osculating spaces 
have dimension  $4$, is the one given by the configuration
\[\A=\{(0,0),(1,0),(0,1),(2,1),(1,2),(2,2)\}.\]
As is also observed in \cite[Cor. 4.4, p.~361]{LM}, this surface is $2$-selfdual. There is a unique conic through 
the six points in $\A$, given by the vanishing of $q_1(x,y)=x^2-xy+y^2-x-y$, 
and the conic $\{q_1=0\}$ does not pass through any other lattice points.

Note that the interior lattice point of the hexagon is omitted. This corresponds to the fact that the surface is the 
(toric) projection of a projectively normal surface in $\mathbb P^6$, namely the Del Pezzo surface of degree $6$. 
The center of projection is a point that contains all hyperplanes that are 2-osculating. It follows that the lattice point 
configuration of the Del Pezzo surface is not $2$nap, so this surface is not $2$-selfdual in our sense, since its 
second dual variety is contained in a hyperplane. Because the projection map is an isomorphism and the Togliatti 
surface is $2$-selfdual, the Del Pezzo surface is considered to be $2$-selfdual in \cite[Thm. 3.4.1, p.~357]{LM}.

\begin{remark}
As a complement to this example, let us recall that the study of varieties with ``too small'' osculating spaces goes back to
 C. Segre \cite{Seg1907} for surfaces,
 to Sisam \cite{Sis1911} for threefolds, and for varieties of any dimension to Terracini \cite{Ter1912}, and subsequently Togliatti. 
It was Terracini who coined the expression ``satisfying [a certain number of] Laplace equations.'' 
Recent work on generalizations of Togliatti's examples includes \cite{F-I02}, \cite{Ila06}, and \cite{V06}. 
The study of these varieties have further been linked to an algebraic notion called the ``weak Lefschetz'' property \cite{MMO13}. 
A different kind of varieties with ``too small'' osculating spaces were discovered and studied by Dye \cite{Dye92}. 
In  \cite{Lv16} Lvovski considers $2$-selfdual space curves and, more generally, 2-selfduality for 
$(n-1)$-dimensional varieties $X\subset \mathbb P^{2n-1}$ satisfying $c_2=1$. He is particularly interested 
in the case when $X$ is Legendrian with respect to a contact structure on  $ \mathbb P^{2n-1}$.
\end{remark}

Consider  now the  configuration  
\[\A':=\{(0,0),(1,0),(0,1),(3,1),(1,2)\}.\]
 The unique conic through these five points, given by the vanishing of 
$q(x,y)=x^2-2xy+2y^2-x-2y$, also goes through the lattice points $(3,3)$, $(4,3)$ 
and $(4,2)$. Thus, adding any one of these three points to $\A'$  gives examples 
of $2$-selfdual surfaces  in $\mathbb P^5$ that are non-smooth.

If we add more than one point, then the surface cannot be 2-selfdual. 
Because if it were, then we would have $c_2\geq 2$, 
but this is not possible since the surface is not 2-Cayley. 
However, if we add all three points, we get a 3-selfdual surface  \cite{M95}.

The polynomial $q$ defines a conic that passes through eight lattice points.
In fact, it is difficult in general to get integer polynomials with many integer roots, see e.g.~\cite{RVVZ01} 
and the references therein. We can extract for instance the following simple example from their constructions.  
Consider three integer numbers $m_1, m_2, m_3$ and consider the univariate 
polynomial $f(x) = \prod_{i=1}^3 (x-m_i)$. 
Set $q_2(x,y) :=(f(x) - f(y))/(x-y) \in \Z[x,y]$, which has degree two. Thus, $q_2$ vanishes at the six
 lattice points $(m_i, m_j), j \not=i$, while $\binom{2+2}2$ is also equal to six. 
The configuration $\A$ given by these six points is
not $2$-jet spanned; in fact, the rank of $A^{(2)}$ equals five, and so $c_1=1$ and
$\A$ is $2$-self dual because it is $2$nap. In fact, when $\{m_1, m_2, m_3\} = \{0,1,2\}$
we get a reflexion of the lattice configuration defining the hexagon in Togliatti's example.
Indeed, $q_2(2-x,y)=q_1(x,y)$, where $q_1$ defines the conic in the Togliatti example.

\subsection{Other smooth non-complete examples}

In his BA Thesis \cite{M95} Mulliken studies (higher) selfdual toric varieties and links them to the property 
that the lattice set is centrally symmetric.  A lattice set $\A:=\{a_1, \ldots, a_N\}\subset \mathbb Z^n$ is 
centrally symmetric if it is symmetric with respect to the midpoint $m:=\frac{1}{|\A|}\sum_i a_i$, i.e.,  
if $a\in \A$ if and only if $2m-a\in \A$. Mulliken's definition of (higher) selfduality is stronger than ours. 
For example, the non centrally symmetric lattice configuration
\[ \A=\{(0,0),(1,0),(1,1), (0,2)\} \]
gives a toric variety rationally parameterized as
\[(t_1,t_2) \mapsto (1:t_1:t_1t_2:t_2^2), \, t_1, t_2 \not=0.\]
Its dual variety is given by the parameterization with exponents in $-\A$ as in~\eqref{eq:actiondual}
\[(t_1,t_2)\mapsto (-1:2t_1^{-1}: -2t_1^{-1}t_2^{-1}:t_2^{-2}), \, t_1, t_2\not= 0,\]
or calling $s_i=t_i^{-1}$,
\[(s_1,s_2)\mapsto (-1:2s_1: -2s_1s_2:s_2^{2}), \, s_1, s_2\not= 0,\]
with weights in $\A$.
The dual variety is certainly projectively equivalent to the original variety, 
hence $X_\A$ is $1$-selfdual according to our definition (see Figure \ref{fig:1}).
But  it is not self dual in Mulliken's terminology. 

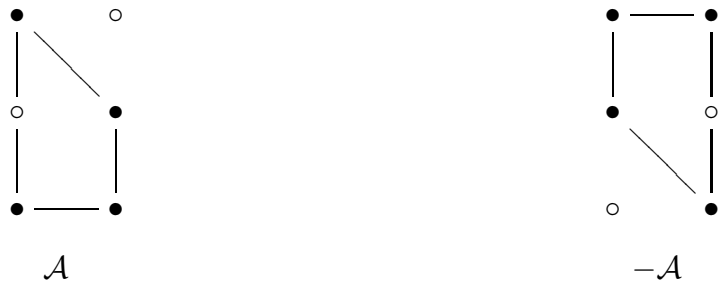
\begin{figure}[h]
 \begin{multicols}{2}
 \centerline{
\xymatrix{
 {\bullet}\ar@{-}[d]\ar@{-}[dr]&{\circ} \\
    {\circ}\ar@{-}[d]&{\bullet}\ar@{-}[d] \\
     {\bullet}\ar@{-}[r]&{\bullet}\\ 
}
}
\bigskip

\centerline{$\mathcal A$ }

\centerline{
\xymatrix{
 {\bullet}\ar@{-}[r]\ar@{-}[d]&{\bullet}\ar@{-}[d] \\
   {\bullet}\ar@{-}[dr]&{\circ}\ar@{-}[d]\\
     {\circ}&{\bullet}\\ 
}
}
\bigskip

\centerline{$-\mathcal A$}
  \end{multicols}
  \caption{A non centrally symmetric selfdual configuration.}
  \label{fig:1}
\end{figure}

\begin{remark}
 Mulliken notices in  \cite[1.5.2]{M95} an ``unexplained'' symmetry  of the matrices $A^{(k-i)}$ and $\Ker A^{(i)}$ 
for a $k$-selfdual variety \cite[p.~9]{M95}. This has a natural explanation: the $k$-selfduality implies
 that $\Ker A^{(i)}$ is the transpose of the $(k-i)$th 
jet matrix of the dual variety, hence equal to $A^{(k-i)}$ because of the $k$-selfduality. 
 More precisely, in the central symmetric case, if $X=X_\A$ is $k$-selfdual, then $\Ker A^{(k)}$ has rank $1$. The map from 
 $X_\A$ to $X^{(k)}=X_\A^{(k)}$ is given by $(\mathcal O_X^{(N+1)})^\vee \to (\Ker j_k)^\vee=:\mathcal L'$. 
We also have (generically) exact sequences (cf. \cite{P83})
 \[0\to (\mathcal P^{k-i}_{X^{(k)}}(\mathcal L'))^\vee \overset{(j_{k-i})^\vee}
\longrightarrow \mathcal O_X^{(N+1)}\overset{j_i}\longrightarrow \mathcal P^i_X(\mathcal L) \to 0. \]
 At the general point, the map $j_i$ to the right is given by the matrix $A^{(i)}$ 
and the left map $(j_{k-i})^\vee$ by the transpose of the matrix $A^{(k-i)}$, 
 which, because of exactness of the sequence, is equal to $\Ker A^{(i)}$.
 \end{remark}

Mulliken proposed the following $2$-parameter family of centrally symmetric surfaces:
\[\A=\{(0,0),(1,0),(0,1),(2,1),(c-1,e),(c,d-1),(c,d),(c-1,d),(c-2,d-1),(1,d-e)\},\]
where $(c,d,e)=(j,1+(j-2)m,(j-3)m)$, where $j\ge 5$, $m\ge 1$.
For $j= 5$, $m= 1$, this is a hexagon with four interior lattice points, 
and it defines a smooth surface. For $j\ge 6$, $m=1$, 
we get an octagon with two interior lattice points, which defines a smooth surface.
For $j\ge 5$, $m\ge 2$, all ten lattice points are vertices, 
so we get decagons with no interior lattice points, and the surfaces are smooth.
Mulliken conjectured that these surfaces are $3$-sefdual, and checked it for  $j=5$ and $m\ge 2$.  
In fact, we checked that for any choice of integers $(c,d,e)$ with $d\neq 1$,
$d\neq 2 (c-1)$, the configuration $\A$ is $3$nap and $c_3=1$. Therefore, $X_\A$ is $3$-selfdual, 
again by item~(i) in Proposition~\ref{prop:rCayley} (but not smooth for general choices of $(c,d,e)$).

Perkinson's octagon in \cite[Thm.~3.2, (4), p.~493]{Pe00} gives another 
example of a $3$-selfdual smooth surface \cite[1.5.2, p.~9]{M95}. If we remove any 
two non-adjacent points, then the resulting hexagon gives a $2$-selfdual, 
possibly non-smooth and possibly non centrally symmetric, surface (see e.g. Figure \ref{fig:2}).

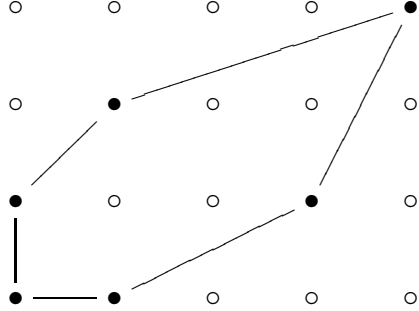
\begin{figure}[h]
\centerline{
\xymatrix{
  {\circ}&{\circ}&{\circ} &{\circ}&{\bullet} \\
  {\circ}&{\bullet}\ar@{-}[urrr]&{\circ} &{\circ}&{\circ} \\
   {\bullet}\ar@{-}[ur]&{\circ}&{\circ} &{\bullet}\ar@{-}[uur]&{\circ} \\
    {\bullet}\ar@{-}[u] \ar@{-}[r] &{\bullet}\ar@{-}[urr]&{\circ} &{\circ}&{\circ} \\
   }}
   \caption{A non centrally symmetric, non smooth 2-selfdual configuration.}
   \label{fig:2}
   \end{figure}
\medskip

Perkinson's dodecagon in  \cite[Thm.3.2, (5), p.~493]{Pe00} is an example of a smooth $5$-selfdual surface. 
Note also that in the two and  three dimensional smooth examples of Perkinson, 
with $c_k=1$, some (interior) lattice points in the convex hull are always omitted 
(thus $\A$ is not complete).

The dotted lattice points in the ``incomplete'' square in Figure~\ref{fig:3} define a configuration $\A$, 
which gives an example of  a $4$-selfdual smooth toric surface which is not centrally symmetric. 
The projective variety $X_\A$ is smooth because at each of the four vertices, we have then 
neighboring lattice points in both directions. We computed that $\A$ is $4$nap, and  $\rk A^{(4)}=15$, 
so that $c_4=\dim \Ker A^{(4)}=1$ and $X_\A$ is $4$-selfdual by item (i) in Proposition~\ref{prop:rCayley}.
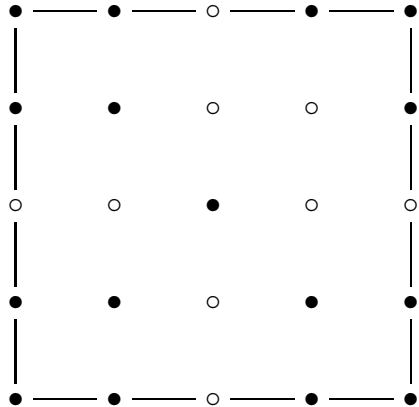
\begin{figure}[h]
\centerline{
\xymatrix{
  {\bullet}\ar@{-}[r] &{\bullet}\ar@{-}[r]&{\circ} \ar@{-}[r]&{\bullet}\ar@{-}[r]&{\bullet}\ar@{-}[d]  \\
  {\bullet}\ar@{-}[u] &{\bullet}&{\circ} &{\circ}&{\bullet}\ar@{-}[d] \\
   {\circ}\ar@{-}[u] &{\circ}&{\bullet} &{\circ}&{\circ}\ar@{-}[d] \\
    {\bullet}\ar@{-}[u] &{\bullet}&{\circ} &{\bullet}&{\bullet}\ar@{-}[d] \\
     {\bullet}\ar@{-}[u]\ar@{-}[r] &{\bullet}\ar@{-}[r]&{\circ} \ar@{-}[r]&{\bullet}\ar@{-}[r]&{\bullet}\\ 
   }}
   \caption{A non centrally symmetric smooth 4-selfdual configuration.}
   \label{fig:3}
   \end{figure}
   \medskip

\subsection{Smooth complete examples}\label{ssec:emilia}

The following $3$-selfdual complete toric surface occurs in~\cite{MM16} and~\cite{Betal15} under (affinely equivalent) different
disguises. It corresponds to the blow-up at two points of 
the Veronese-Segre embedding of ${\mathbb P}^1 \times
{\mathbb P}^1$ polarized by ${\mathcal O}(3,2)$.  Both configurations are depicted in Figure~\ref{fig:32}. 
We will show in Theorem~\ref{th:complete} that all Veronese-Segre embeddings
in any dimension are (smooth, complete) selfdual.

\begin{figure}[h]
 \begin{multicols}{2}
 \centerline{
\xymatrix{
 {\bullet}\ar@{-}[d]\ar@{-}[r]&{\bullet}\ar@{-}[r]&{\bullet}\ar@{-}[dr]&{\circ} \\
  {\bullet}\ar@{-}[dr]&{\bullet}&{\bullet}&{\bullet}\ar@{-}[d] \\
    {\circ}&{\bullet}\ar@{-}[r]&{\bullet}\ar@{-}[r]&{\bullet} \\
}
}

\centerline{
\xymatrix{
 {\bullet}\ar@{-}[r]\ar@{-}[d]&{\bullet}\ar@{-}[dr]&{\circ}&{\circ} \\
   {\bullet}\ar@{-}[dr]&{\bullet}&{\bullet}\ar@{-}[dr]&{\circ}\\
     {\circ}&{\bullet}\ar@{-}[dr]&{\bullet}&{\bullet}\ar@{-}[d]\\ 
     {\circ}&{\circ}&{\bullet}\ar@{-}[r]&{\bullet}\\
}
}

  \end{multicols}
  \caption{The example in \S~\ref{ssec:emilia}: As in~\cite{MM16} on the left, as in~\cite{Betal15} on the right.}\label{fig:32}
\end{figure}
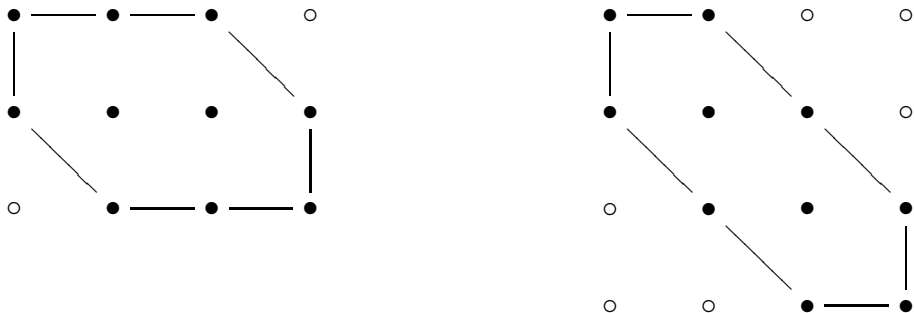


\section{General constructions of higher selfdual toric varieties}\label{sec:general}

We give  some general constructions of $k$-selfdual projective toric embeddings.  

\subsection{In terms of Hilbert functions}

Let $I(\A)$ denote the ideal of the points in $\A=\{(1,a_0), \dots, (1,a_n)\}$. 
The value of the Hilbert function of $I(\A)$ at  $k$ is the codimension in the linear space of 
homogeneous polynomials of degree $k$ in $n+1$ variables  of those polynomials that belong to $I(\A)$. 
It equals the affine Hilbert function $H_{I(\A)}(k)$, which
is the codimension in the linear space $\Q[x]_{\le k}$ 
of polynomials in $n$ variables of degree at most $k$, of those 
polynomials vanishing on $\{a_0, \dots, a_N\}$. 
Thus, for any $k$, ${I(\A)}(k) $ can be identified with the left kernel of $A^{(k)}$ 
(by the proof of Lemma~\ref{lem:knap}, cf. also Proposition~1.1 in~\cite{Pe00}), and so
$H_{I(\A)}(k) =\binom{n+k}{k} - \rk A^{(k)} = \binom{n+k}{k} +1 - d_k$.

The following proposition gives a criterion for a configuration $\A$ which is not $2$-Cayley to be $k$-selfdual.

\begin{prop}\label{prop:Hilbert}
Given points $a_0,\ldots, a_N \in \mathbb Z^n$ such that  $\A:=\{a_0, \ldots, a_N\}$ is 
not $2$-Cayley. Let $k\ge 1$ be given. Then $X_\A$ is $k$-selfdual if and only if 
for all $j\in \{0,\ldots,N\}$,
\begin{equation}\label{eq:H}
H_{I(\A)}(k)=H_{I(\A\setminus \{(1,a_j)\})}(k)
= N.
\end{equation}
Equivalently,
both $\{ Q\in \mathbb Q[x]_{\le k} |\, Q(a_i)=0, i=0,\ldots,N\}$ and
$\{ Q\in \mathbb Q[x]_{\le k} |\, Q(a_i)=0, i=0,\ldots,N, i\neq j\}$ 
(for any $j\in \{0,\ldots,N\}$), have dimension equal to
$\binom{n+k}k -N$.
\end{prop}

\begin{proof}
Since $\A$ is not $2$-Cayley, we know by Theorem~\ref{th:dim} that $X_\A$ 
is $k$-selfdual if and only if $\A$ is $k$nap and the kernel of 
$A^{(k)}$ has rank $c_k=1$. Then, 
since $A$ is a matrix of size $\binom{n+k}k \times (N+1)$, $\rk A^{(k)} =N$ and thus
$H_{I(\A)}(k)=\binom{n+k}k-N$. By Lemma \ref{lem:knap} (c),  $\A$ is $k$nap if 
$H_{I(\A)}(k)=H_{I(\A\setminus \{(1,a_j)\})}(k)$ for all $j=0,\ldots,N$.

This can be equivalently phrased in terms of the dimension of the left kernels 
of $A^{(k)}$ and its minor corresponding to the deletion of any column 
$j$, as in the second part of the statement.
\end{proof}

In particular, we deduce:

\begin{cor}\label{cor:general}
Any choice of $\binom{n+k}k+1$ \emph{general} points
 in $\mathbb Z^n$ gives a configuration $\A$ such that $X_\A$ is $k$-selfdual.
\end{cor}

\begin{proof}
Let $\A$ be any general configuration of  $\binom{n+k}k+1$ (lattice) points. 
It is well known that if we take away any point, the (general) points in $\A \setminus \{(1,a_j)\}$
 impose $\binom{n+k}{k}$ number of independent conditions on polynomials of
degree $k$ passing through them, that is, that $H_{I(\A)}(k) = \binom{n+k}k$
 (this goes back at least to Castelnuovo~\cite{H60}).  This means that 
any maximal minor of $A^{(k)}$ is non-zero. Thus, $c_k=1$ 
and the left kernel of $A^{(k)}$ and of any of its maximal minors
is zero. The result follows from Proposition~\ref{prop:Hilbert}. 
\end{proof}

\begin{remark}\label{rem:general}
For the convenience of the reader, we include a simple proof that for any general configuration of
$\binom{n+k}{k}$ points there is no non-zero polynomial of degree at most $k$ which vanishes on them.
It is enough to show that this is the case for a particular configuration.
We consider $\A$ equal to the lattice points in the 
standard simplex of size $k$ in $n$ dimensions. So, both the rows and
the columns of $A^{(k)}$ are indexed by $\Delta_k=\{\alpha \in \N^n \, | \, |\alpha| \le k\}$.

For any $\alpha \in \Delta_k$,  denote by $m_\alpha$ the polynomial
\begin{equation}\label{eq:ma}
m_\alpha \, = \prod_{\alpha_i > 0} \frac{x_i (x_i-1) \dots (x_i-\alpha_i +1)}{\alpha_i!}. 
\end{equation}
Note that $m_\alpha(\alpha)=1$ for any $\alpha$ and $m_\alpha(\beta) =0$ whenever there exists
an index $i$ with $\beta_i > \alpha_i$.
We order these polynomials by
ordering their indices $\alpha$   as in Definition~\ref{def:Ak}.
We similarly define the vectors $\mb w_\alpha \in \Z^{N+1}$, obtained as the \emph{coordinatewise evaluation} of 
$m_\alpha$ at the points in $\Delta_k$.  Let $A_m^{(k)}$ be the
 $\binom{n+k}k \times (N+1)$ integer matrix with rows $\mb w_\alpha, \alpha \in \Delta_k$,
 where we order the columns and the rows with
the same ordering. It is straightforward to check that $A_m^{(k)}$ is an upper triangular matrix with $1$'s along
the main diagonal and therefore it has maximal rank $\binom{n+k}k$. 
Moreover, there exists an invertible upper triangular matrix
$M$ with $1$'s along the main diagonal such that $M \, A^{(k)}= A_m^{(k)}$ and so
the rank of $A^{(k)}$ is also maximal.
\end{remark}

\subsection{Subconfigurations of {k}-selfdual configurations}

The following result,
which follows from  Theorem~\ref{thm:eL}, shows
how to find $k$-selfdual subconfigurations of a given $k$-selfdual configuration $\A$. 
It extends Proposition~4.20 in \cite{BDR11} to
any $k \in \N$ and can be proved similarly. 

\begin{prop}\label{subset}
Assume $X_\A$ is $k$-selfdual.
Let $\D\subseteq \A$ be an arbitrary subset of $\A$. 
Then, either $\D$ is not $k$nap (i.e., $D^{(k)}$ is a pyramid), 
or $X_\D$ is $k$-selfdual.
\end{prop}

When a configuration $\A$ is not $k$nap, the following lemma  translates to our setting Theorem~2.2.1 in~\cite{M95}.

\begin{lemma}\label{lem:DJ}
Assume a lattice configuration $\A$ with $c_k=1$ is not $k$nap. 
Let $J$ denote the set of indices of zero coordinates of the elements in
$\Ker A^{(k)}$ and define $\D_J = \{(1,a_i)\, : \, i\notin J\}$. Then, $X_{\D_J}$ is $k$-selfdual.
\end{lemma}

The proof is straightforward. In fact, $\D_J$ is $k$nap and the kernel of the associated matrix
$D_J^{(k)}$ has dimension $1$, as it is generated by the vector of nonzero coordinates of a generator of $\Ker A^{(k)}$.
One example of the use of this lemma is the Togliatti surface we recall in \S~\ref{ssec:Togliatti}. 

\subsection{Joins of selfdual configurations}

Another way of constructing $k$-selfdual configuration is by the projective
 join of two (or more) such configurations, a particular case of Cayley configurations.
Recall the definition of the join of two varieties over a field $\K$:

\begin{definition} \label{def:join}
Let $V_1, \dots, V_s$ be  finite dimensional $\K$-vector spaces 
and let
$X_1 \subseteq \bP(V_1)$, \dots, $X_s \subseteq \bP(V_s)$ be
 projective varieties. The  \emph{join}  of  $X_1, \dots,
X_s$ is the projective subvariety of $\bP(V_1 \oplus \cdots \oplus V_s)$
defined by
 \[
\opJ(X_1, \dots, X_s)= \overline{\bigl\{[x_1:\cdots :x_s]\, |\,[x_i]
\in X_i \bigr\}}.
\]
\end{definition}

We have  $\dim
\opJ(X_1,\dots,X_s)=\sum \dim X_i+s-1$. Note that, in the ``trivial'' case when $X_i=\mathbb P(V_i)$ for all $i$, 
then $\opJ(X_1,\dots,X_s)=\bP(V_1 \oplus \cdots \oplus V_s)$, 
and that in all other cases (with $s\ge 2$), $\opJ(X_1,\dots,X_s)$ 
is singular at all points of the embedded varieties $X_i\subset \mathbb P(V_i)\subset  \bP(V_1 \oplus \cdots \oplus V_s)$. 
This last fact can be seen 
by using the Jacobian criterion on the generators of the ideal defining $\opJ(X_1,\dots,X_s)$.

Given  projective toric embeddings $X_{\A_i}\subseteq \mathbb P^{N_i}$, $i=1, \dots, s$,
their join  is also a toric variety. Assume $\A_i \subset \Z^{n_i}$,
and consider the configuration
$\A=\A_1\times\{0\} \times \dots \times \{0\}\cup
\{0\}\times \A_2\times\{0\} \times \dots \times \{0\}\cup \dots \subset \Z^{n_1+\cdots +n_s+s} $.
The projective toric variety associated to $\A$ is  the join
$X_\A=\opJ(X_{\A_1},\dots, X_{\A_s})$. The matrix $A$ associated with $\A$ has then the block form
\[A=\left(\begin{array}{cccccc}\label{arr:B}
A_1 &   0  &   0  &   \cdots   &  0&   0\\
0   &  A_2 &  0 &  \cdots  &  0&  0\\
\vdots & \vdots&\vdots & & \vdots & \vdots \\
0   &  0  & 0 & \cdots &  A_{s-1} &   0\\
0   &  0  & 0 & \cdots & 0 &   A_{s}
\end{array}\right),
\]
where $A_i$ is the matrix associated with $\A_i$.
Note that the configuration of the join $X_\A$ is $s$-Cayley. 
\medskip

The following proposition provides examples of $k$-selfdual configurations
for any value of  $\dim \Ker A^{(k)}$. 

\begin{prop}\label{prop:jk}
Assume $\A_1, \dots, \A_s$ are $k$nap and $k$-selfdual. 
Then the join $X_\A=\opJ(X_{\A_1},\dots,X_{\A_s})$ is $k$nap and $k$-selfdual, 
with 
$$\dim \Ker A^{(k)}=\dim \Ker A_1^{(k)}+\cdots +\dim \Ker A_s^{(k)}\ge s.$$
\end{prop}

\begin{proof} 
The proof of Proposition~\ref{prop:jk} follows from Lemma~\ref{lem:knap} and
the characterization in Theorem~\ref{thm:eL}, since
the kernel of $A^{(k)}$ is the direct sum of the kernels of the associated matrices $A_1^{(k)},\dots, A_s^{(k)}$.
\end{proof}
 
\begin{example}\label{ex:curves}
Let $\A_1=\{(1, a_0),\dots, (1,a_{k+1})\}$ and $\A_2=\{(1,a'_0), \dots, (1, a'_{k+1})\}$, where
$a_0<\cdots < a_{k+1}$ and $a'_0<\cdots < a'_{k+1}$.
If the elements in  $\A_1$ are coprime, then the associated toric variety
$X_{\A_1}$ is a rational curve  of degree $a_{k+1}-a_0$, which
is smooth if and only if $a_1-a_0=a_{k+1}-a_k=1$ (and similarly for $\A_2$).
The configurations ${\A_1}, {\A_2}$ are  $k$nap  with $\dim \Ker A_i^{(k)}=1$. 
Hence they are $k$-selfdual by Proposition~\ref{prop:rCayley} (i).  
(This follows from Theorem~\ref{th:dim} (b) as well, since the $k$th dual 
of a (non degenerate) curve in $\mathbb P^{k+1}$ has dimension 1.) 
Their  join $X_\A=\opJ(X_{\A_1},X_{\A_2})$ is  a $k$-selfdual 
threefold by Proposition~\ref{prop:jk}, with $\dim \Ker A^{(k)}=2$.

Geometrically, this situation can be explained by the fact that the $k$-osculating spaces to $X_\A$ 
at any point on  the line joining a point in $X_{\A_1}$ 
and a point in $X_{\A_2}$ (but not on $X_{\A_1}$ or $X_{\A_2}$) 
 is equal to the join of the $k$-osculating spaces to $X_{\A_1}$ and $X_{\A_2}$ at those points. 
So each point on this line corresponds to points in a linear space in the dual space, and vice versa, 
each $k$th osculating space is $k$-osculating at all points on a line.
\end{example}
 
\subsection{Cayley configurations}

We present  a family of toric $k$-selfdual examples which are not joins, but for which the dimension $c_k$ of the
kernel of $A^{(k)}$ can be arbitrarily high.

\begin{example} \label{ex:cayleyP}
Let  $r\ge 2$, and let $d_1\le \dots \le d_r$ be positive integers. Let $\A_i=\{0,1,\ldots,d_i\}$ 
for $i=1,\dots,r$ be the configuration of lattice points in the polytope $d_i [0,1]$. Now
consider the $r$-Cayley configuration 
$\A=\text{Cayley} (\A_1,\ldots, \A_r)$. Then $\A$ defines a rational normal scroll $X_\A$ and is $k$nap if $k\le d_1$. 
If $k= d_1$, then $X_\A$ is a $k$-selfdual toric variety if and only 
if $k=d_1=\cdots =d_r$, i.e., if and and only if $X_\A$ is a \emph{balanced} rational normal scroll.
This follows from the results of \cite{PS84}, see also Proposition~4.1 in \cite{DDP11}. 
Note that $c_k= r-1$
can take any value as  $r$ varies. Note also that in this case all $\A_i$ 
are equal, and so the toric variety associated to the Cayley configuration is the product
$\pr{r-1} \times X_{\A_1}$.
\end{example}

We have the following general result. 

\begin{prop}\label{prop:product}
Let $k, r \ge 2$. Consider a lattice configuration $\B \subset \Z^d$ of cardinality 
$m +1$ such that the general $k$th osculating space of $X_{\B}$ is
the whole ${\mathbb P}^m$ and $\dim \Ker B^{(k-1)}= c_{k-1}(B)=1$.
Call $\A= \Cayley(\B, \dots, \B)$ ($r$ times), so that
$X_\A = {\mathbb P}^{r-1} \times X_\B\subset \mathbb P^{r(m+1)-1}$.
Then, $X_\A$ is $k$-selfdual if and only if $X_\B$ is $(k-1)$-selfdual.
\end{prop}

Note that $X_\A$ is smooth if and only if $X_\B$ is smooth.

\begin{proof}
Let $B \in \Z^{(d+1)\times (m+1)}$  denote the matrix of $\B$.  The matrix
\[\left(\begin{array}{cccccc}
B^{(k-1)} &   0  &   0  &   \cdots   &  0&   0\\
0   &  B^{(k-1)} &  0 &  \cdots  &  0&  0\\
\vdots & \vdots&\vdots & & \vdots & \vdots \\
0   &  0  & 0 & \cdots &  B^{(k-1)} &   0\\
 u_1B^{(k)}  & u_2B^{(k)} & \cdots  & \cdots  &  u_{r-1}B^{(k)}  &  B^{(k)}
\end{array}\right)
\]
determines the $k$th osculating space to $X_\A$ at a point of a (general) ruling, 
where $(u_1: \cdots :u_{r-1}:1)\in \mathbb P^{r-1}$ parameterizes the points of the ruling. 
It follows that we can write 
\begin{equation}\label{arr:B1}
A^{(k)}=
\left(\begin{array}{cccccc}
B^{(k-1)} &   0  &   0  &   \cdots   &  0&   0\\
0   &  B^{(k-1)} &  0 &  \cdots  &  0&  0\\
0   &  0  & 0 & \cdots &  B^{(k-1)} &   0\\
0   &  0  & 0 & \cdots & 0 &   B^{(k-1)}\\
 B^{(k-1,k)}  & B^{(k-1,k)} & \cdots  & \cdots  &  B^{(k-1,k)}  &  B^{(k-1,k)}
\end{array}\right)
\end{equation}
where $B^{(k-1,k)}$ is equal to the matrix obtained by removing $B^{(k-1)}$ from $B^{(k)}$. 


We identify $\Ker A^{(k)}$ in terms of $\Ker B^{(k)}$.
The hypothesis on the $k$th osculating spaces of $X_\B$ translates to
$\rk B^{(k)} = m+1$, this is the rank of the matrix with rows in $B^{(k-1)}$ and in $B^{(k-1,k)}$. 
Also, as $c_{k-1}(B)=1$, we have that $\rk B^{(k-1)} = m$.
Let $v \in \Z^{m+1}$ be a generator of $\Ker B^{(k-1)}$.
It is clear that all the vectors of the form $(v,0,\dots, 0, -v, 0, \dots , 0)$
lie in $\Ker A^{(k)}$. Denote by $V$ the vector space they generate, 
which has dimension $r-1$.
We claim that our hypotheses imply $V = \Ker A^{(k)}$. We check that they have
the same dimension. In fact,
$A^{(k)} \in \Z^{\binom{n+k}k \times r(m+1)}$ has rank equal to $\rk B^{(k)} +
(r-1) \rk B^{(k-1)} =m+1 +(r-1) m = rm +1$. Then $c_k(A) = r(m+1) - rm -1  = r-1$, as wanted.

We need to check that the conditions in Theorem~\ref{thm:eL} hold for $A$ 
if and only if they hold for $B$. Indeed, the columns of the matrix 
\begin{equation}\label{arr:A1}
\left(\begin{array}{rrrr}
v &   v &    \cdots   &    v\\
-v  &  0&   \cdots  &  0\\
0   &  -v &  \cdots &  0\\
\vdots & \vdots & \vdots  & \vdots \\
0   &  0  &  \cdots & -v\\
\end{array}\right)
\end{equation}
give a basis of $\Ker A^{(k)}$.  
It is clear that the first $m+1$ row vectors lie in the line $L_1$ generated by
$(1, \dots, 1)$, the next $m+1$ row vectors lie in the line $L_2$ generated by
$(1, 0, \dots, 0)$, and so on. The last $m+1$ row vectors lie in the line $L_r$ generated
by $(0, \dots, 0, 1)$.  As $\A= \Cayley(\B, \dots, \B)$,  the corresponding vectors $e_{L_i}$ lie in the row span of $A$.  
Then, $\A$ is $k$-selfdual if and only if $\A$ is $k$nap.  
On the other side, $\B$ is $(k-1)$-selfdual if and only if it is $(k-1)$nap because
$c_{k-1}(B)=1$.

Now, if $\A$ is $k$nap,  $A^{(k)}$ is not a pyramid,
 and clearly $B^{(k-1)}$ cannot be a pyramid,
that is, $\B$ is $(k-1)$nap.  Reciprocally, if $\B$ is $(k-1)$nap and $r\ge 2$, 
it follows from the shape of $A^{(k)}$ in~\eqref{arr:B1} that $\A$ is $k$nap.
\end{proof}

\begin{example}\label{ex:Valles}
Taking $r=2$, we can iterate the construction in~Proposition~\ref{prop:product}
 to get the following example in~\cite{V06}. In this case, the dimension of the kernel does not increase as $k$ increases.
As shown by Vall\`es in Lemma~ 4.1 in \cite{V06},  the Segre embedding $(\mathbb P^1)^n \to \mathbb P^{2^n-1}$
 is $(n-1)$-selfdual. In fact, the embedding is defined by all monomials of degree at most $n$ in variables $x_1,\ldots,x_n$
 with no repeated factor, i.e., by the lattice points $\mathcal A$ equal to the vertices of the unit $n$-cube. 
Since ${\mathbb P^1} \times \mathbb P^1$ is $1$selfdual with $c_1=1$,
it follows by induction writing
$(\mathbb P^1)^n= {\mathbb P^1} \times (\mathbb P^1)^{n-1}$ and using 
Proposition~\ref{prop:product}, that $\A$ is $(n-1)$-selfdual (with $c_{n-1}=1$).
\end{example}

\subsection{Using the Cayley-Bacharach theorem} \label{ex:CB}

We give a family of examples where selfduality is a consequence of the well known 
Cayley--Bacharach theorem~\cite[Thm. CB6]{EGH} (see~\cite{GH} for a proof
using the Euler--Jacobi vanishing condition for the associated residues).

\begin{thm}[Cayley-Bacharach]\label{th:CB}
Given $d_1, \dots, d_n \ge 1$, define the critical degree $\rho=\sum_{i=1}^m d_i - (n+1)$.
Assume $H_1, \dots, H_n$ are hypersurfaces in $\mathbb P^n$ of respective degrees $d_1, \dots, d_n$ intersecting
in a configuration $C$ of  $d_1 \cdots d_n$ points (necessarily simple).  If a hypersurface $H$ of degree
$\rho$ contains all the points in $C$ but one, then $C \subset H$.
\end{thm}

The following Lemma is an easy consequence of Theorem~\ref{th:CB}.

\begin{lemma} The following classical configurations are selfdual.
\begin{enumerate}
\item[(1)] Consider two cubic forms $C_1, C_2 \in \mathbb Z[x_1,x_2]$ such that their
intersection is equal to nine lattice points $a_0, \ldots, a_8 \in \Z^2$. 
Then,  $\A_C=\{a_0,\ldots,a_8\}$ gives a $3$-selfdual surface. 
\item[(2)] Consider three quadrics $Q_1, Q_2,Q_3\in \mathbb Z[x_1,x_2,x_3]$ such that their intersection is equal 
to eight lattice points $a_0,\ldots,a_7\in \Z^3$. Then $\A_Q=\{a_0,\ldots,a_7\}$ gives a $2$-selfdual threefold. 
\end{enumerate}
\end{lemma}

\begin{proof}
We first check that condition~\eqref{eq:H} in Proposition~\ref{prop:Hilbert}
is satisfied. Note that by our assumption about the form of $\A_C$, the left kernel
of $A_C^{(3)}$ has dimension $2= \binom{2+3}{3}-8$.
  In the second case, the left
kernel of $A_Q^{(2)}$ has by hypothesis dimension $3 = \binom{3+2}{2}-7$.

Note that in both cases, the respective critical degrees are
$3-1+3-1-1=3$ and $2-1 + 2-1+ 2-1 -1=2$. So, by the Cayley--Bacharach theorem 
any  cubic through eight of the points in $\A_C$ also passes through the remaining point 
(this is originally a result by Chasles) and respectively, any quadric through seven of the points in $\A_Q$ 
also passes through the remaining point. Thus, condition~\eqref{eq:H} in 
Proposition~\ref{prop:Hilbert} is also satisfied. 
 We deduce that
that $X_{\A_C}$ is $3$-selfdual and $X_{\A_Q}$ is $2$-selfdual. 
\end{proof}

In this same framework, we get the following general result.

\begin{thm}\label{th:complete}
Let $\A$ be a lattice point configuration such that $X_\A$ is 
equal to a Segre--Veronese embedding of the following forms:
\begin{equation}\label{eq:nor}
 \mathbb P^1 \times \cdots \times \mathbb P^1 \hookrightarrow \mathbb P^N, 
\end{equation}
or
\begin{equation}\label{eq:or}
\mathbb P^{r-1} \times \mathbb P^1 \times \cdots \times \mathbb P^1 
\hookrightarrow \mathbb P^N,  \quad r \ge 3,
\end{equation}
where there are $m\ge 1$ copies of $\mathbb P^1$'s, and the embeddings are of type 
$(\ell_1,\dots,\ell _m)$ in~\eqref{eq:nor} and of type $(1, \ell_1,\dots,\ell _m)$ in~\eqref{eq:or},  with $\ell_i\ge 1$. 
Set $k:=\sum_{j=1}^m \ell_i =k_\ell +1$. Then,  $X_\A$ is a 
smooth, equivariantly embedded toric variety which is $k_\ell$-selfdual in case~\eqref{eq:nor} 
with $c_{k_\ell}=1$, and $k$-selfdual in case~\eqref{eq:or} with
$c_k=r-1$, with $\A$ complete.
\end{thm}

\begin{proof}
We first prove the assertion about the embedding in~\eqref{eq:nor}.
Given $\ell_1, \dots, \ell_n \ge 1$, 
consider the parallelepiped $\Pi_\ell = [0, \ell_1] \times \dots \times [0, \ell_n]$.
Denote by $\A_\ell$ the configuration of lattice points in $\Pi_\ell$, and set $k_\ell :=\ell_1+\dots+\ell_n-1$.
We want to see that $c_{k_\ell}=1$ and  $\A_\ell$ is $k_\ell$-selfdual.

Consider the polynomials $f_i (x) = \prod_{j=0}^{\ell_i}( x_i -j) \in \Q[x_1,\dots, x_n]$, 
where $i$ runs from $1$ to $n$. The homogenized polynomials  define, respectively, hypersurfaces 
$H_i$ in $\mathbb P^n_\mathbb Q$ of degree $d_i=\ell_i +1$. 
The intersection $H_1 \cdots H_n$ coincides with our configuration $\A_\ell$ 
(thought inside projective space) and the critical degree equals $k_\ell$. 
So, by Theorem~\ref{th:CB}, any polynomial of degree at most $k_\ell$
 vanishing at all points in $\A_\ell$ but one, also vanishes at the remaining point. 
This means that $\A_\ell$ is $k_\ell$nap by Lemma~\ref{lem:knap}.  

Note that all the points in $A_\ell$ define monomials with degree at most $k_\ell$, 
except for the ``corner'' point $(\ell_1, \dots, \ell_n)$ which has degree $k_\ell +1$.
Passing to the matrix
$(A_\ell)_m^{(k_\ell)}$ as in Remark~\ref{rem:general}, it is straightforward to see that 
$\rk A_\ell^{(k_\ell)}$ equals the cardinality of $\A_\ell$ minus one, that is, 
that $c_{k_\ell}=1$. It follows that $\A_\ell$ is $k_\ell$-selfdual by item (i) in Proposition~\ref{prop:rCayley}.

The second case in~\eqref{eq:or} follows by Proposition~\ref{prop:product}.
\end{proof}

\end{document}